\documentclass[]{interact}

\usepackage{epstopdf}
\usepackage[caption=false]{subfig}

\theoremstyle{plain}
\newtheorem{theorem}{Theorem}[section]

\theoremstyle{definition}

\theoremstyle{remark}

\begin{document}

\title{An Unbiased Variance Estimator with Denominator $N$}

\author{
\name{Dai Akita\textsuperscript{a}\thanks{CONTACT Dai Akita Email: d.akita@ne.t.u-tokyo.ac.jp}}
\affil{\textsuperscript{a}Graduate School of Information Science and Technology, The University of Tokyo, Japan}
}

\maketitle

\begin{abstract}
Standard practice obtains an unbiased variance estimator by dividing by $N-1$ rather than $N$. Yet if only half the data are used to compute the mean, dividing by $N$ can still yield an unbiased estimator. We show that an alternative mean estimator $\hat{X} = \sum c_n X_n$ can produce such an unbiased variance estimator with denominator $N$. These \emph{average-adjusted unbiased variance} (AAUV) permit infinitely many unbiased forms, though each has larger variance than the usual sample variance. Moreover, permuting and symmetrizing any AAUV recovers the classical formula with denominator $N-1$. We further demonstrate a continuum of unbiased variances by interpolating between the standard and AAUV-based means. Extending this average-adjusting method to higher-order moments remains a topic for future work.
\end{abstract}

\section{Introduction}
\label{sec:intro}

Let $X_1, X_2, \ldots, X_N$ be independent and identically distributed (i.i.d.) observations from a population with unknown mean $\mu$ and variance $\sigma^2$. The naive sample variance is defined by
\begin{equation}
\label{eq:biased-var}
    \bar{\sigma}^2 = \frac{1}{N}\sum_{n=1}^N (X_n - \bar{X})^2,
\end{equation}
where
\begin{equation}
\label{eq:sample-mean}
    \bar{X} = \frac{1}{N}\sum_{n=1}^N X_n.
\end{equation}
As explained in standard statistical textbooks, $\bar{\sigma}^2$ is a biased estimator of $\sigma^2$, specifically
\begin{equation}
    \mathbb{E}[\bar{\sigma}^2] = \frac{N-1}{N}\,\sigma^2.
\end{equation}
To correct this bias, we replace $N$ with $N-1$ in the denominator of Equation \eqref{eq:biased-var}:
\begin{equation}
\label{eq:usual-unbiased-var}
    s^2 = \frac{1}{N-1}\sum_{n=1}^N (X_n - \bar{X})^2,
\end{equation}
so that $s^2$ becomes an unbiased estimator of $\sigma^2$, i.e.\ $\mathbb{E}[s^2] = \sigma^2$.

Given that $N$ is even, consider another mean estimator that uses only half of the observations:
\begin{equation}
\label{eq:half-mean}
    \bar{X}_2 = \frac{2}{N}\sum_{n=1}^{N/2}X_n.
\end{equation}
Then, the sum of squared deviations from $\bar{X}_2$, divided by $N$,
\begin{equation}
\label{eq:half-var}
    \bar{\sigma}_2^2 = \frac{1}{N}\sum_{n=1}^N (X_n - \bar{X}_2)^2
\end{equation}
is, in fact, an unbiased estimator of the variance:
\begin{align}
    & \mathbb{E}\left[ \bar{\sigma}_2^2 \right]
    = \mathbb{E} \left[ \frac{1}{N} \sum_{n=1}^N (X_n - \bar{X}_2)^2\right]
    = \mathbb{E} \left[ \frac{1}{N} \sum_{n=1}^N \left( (X_n - \mu) - (\bar{X}_2 - \mu) \right)^2 \right] \notag \\
    & = \frac{1}{N} \sum_{n=1}^N \mathbb{E} \left[ (X_n - \mu)^2 \right]
    - \frac{2}{N} \sum_{n=1}^N \mathbb{E} \left[ (X_n - \mu)(\bar{X}_2 - \mu) \right]
    + \mathbb{E} \left[ (\bar{X}_2 - \mu)^2 \right] \notag \\
    & = \sigma^2 
    - \frac{2}{N} \sum_{n=1}^N \mathbb{E} \left[ (X_n - \mu)\left(\frac{2}{N} \left( \sum_{m=1}^{N/2} X_m \right)- \mu\right) \right]
    + \mathbb{E} \left[ \left( \frac{2}{N} \left(\sum_{n=1}^{N/2} X_n \right) - \mu \right)^2 \right] \notag \\
    & = \sigma^2 
    - \frac{4}{N^2} \sum_{n=1}^N \mathbb{E} \left[ (X_n - \mu) \sum_{m=1}^{N/2} (X_m - \mu) \right]
    + \frac{4}{N^2} \mathbb{E} \left[ \left(\sum_{n=1}^{N/2} (X_n - \mu) \right)^2 \right] \notag \\
    & = \sigma^2 
    - \frac{4}{N^2} \sum_{n=1}^N \sum_{m=1}^{N/2} \mathbb{E} \left[ (X_n - \mu) (X_m - \mu) \right]
    + \frac{4}{N^2} \sum_{n=1}^{N/2} \sum_{m=1}^{N/2} \mathbb{E} \left[ (X_n - \mu) (X_m - \mu) \right] \notag \\
    & = \sigma^2 
    - \frac{4}{N^2} \sum_{n=1}^{N/2} \mathbb{E} \left[ (X_n - \mu)^2 \right]
    + \frac{4}{N^2} \sum_{n=1}^{N/2} \mathbb{E} \left[ (X_n - \mu)^2 \right]
    = \sigma^2
\end{align}
Hence, $\bar{\sigma}_2^2$ is also an unbiased estimator of $\sigma^2$ despite the fact that the denominator is $N$. One may interpret this as adjusting the mean estimator rather than the denominator. Naturally, this raises the question of whether there exist other mean estimators that yield an unbiased variance estimator when the sum of squared deviations is divided by $N$. We call such estimators average-adjusted unbiased variance (AAUV) and discuss their characteristics in this paper.

\section{Related Work}
\label{sec:related_work}

Although the naive estimator $\bar{\sigma}^2$ with denominator $N$ is a biased estimator, it can be preferable in certain practical scenarios. For instance, when mean squared error is prioritized over unbiasedness \cite{imstat2015kids}, or when adopting maximum likelihood estimation \cite{bishop2006pattern} for normally distributed data. From an educational or intuitive standpoint, some even advocate the use of the $N$-denominator estimator for teaching introductory statistics \cite{rumsey2009lets}. Consequently, it is not always the case that one should favor the unbiased estimator over the biased version. Moreover, \cite{reichel2025} proposes an alternative measure of dispersion that does not rely on explicitly estimating the mean, thereby moving beyond the conventional $N$ versus $N-1$ debate. Such discussions highlight the multifaceted considerations in choosing a variance estimator, motivating further exploration of approaches like AAUV.

A wide variety of research has been conducted on the construction of unbiased estimators. For example, unbiased estimators of higher-order central moments can be derived from $h$-statistics \cite{dwyer1937moments,gerlovina2019computer}, while \cite{ruiz2013optimal} proposes an alternative approach that accomodates various sampling designs. In finite-population settings, the use of auxiliary information can yield unbiased estimators of variance \cite{Isaki01031983,singh2013new,Zaman18042023,ahmad2023enhanced}. \cite{voinov2012unbiased} provides a comprehensive overview of standard methods, such as scale adjustment (e.g., $s^2$), utilization of the Rao--Blackwell--Kolmogorov theorem, and solving integral equations. However, none of these works appear to investigate the approach discussed in this paper, namely ``adjusting the mean estimator itself'' so that the denominator remains $N$ while still achieving unbiasedness. Our notion of AAUV addresses precisely this gap by exploiting a suitable linear combination of observations to preserve unbiasedness with a denominator of $N$.

\section{Average-Adjusted Unbiased Variance}

We now generalize the half-sample approach to allow for any weighted mean estimator
\begin{equation}
\label{eq:weighted-mean}
    \hat{X} = \sum_{n=1}^N c_n \, X_n
\end{equation}
and investigate conditions under which
\begin{equation}
\label{eq:aauv}
    \hat{s}^2 = \frac{1}{N}\sum_{n=1}^N (X_n - \hat{X})^2
\end{equation}
is unbiased. We call such $\hat{s}^2$ an \emph{average-adjusted unbiased variance (AAUV)}. Equation \eqref{eq:aauv} can be expanded as follows: 
\begin{align}
    & \hat{s}^2 = \frac{1}{N} \sum_{n=1}^N \left( X_n - \hat{X} \right)^2
    = \frac{1}{N} \sum_{n=1}^N \left( (X_n - \mu) - \sum_{m=1}^N c_m (X_m - \mu) - \left( 1 - \sum_{m=1}^N c_m \right) \mu \right)^2 \notag \\
    & = \frac{1}{N} \sum_{n=1}^N \left( (1 - c_n) (X_n - \mu) - \sum_{m \neq n} c_m (X_m - \mu) - \left( 1 - \sum_{m=1}^N c_m \right) \mu \right)^2.
    \label{eq:hatmu2}
\end{align}
From the fact that
\begin{equation}
    \mathbb{E}\left[ (X_n - \mu) (X_m - \mu) \right] =
    \begin{cases}
      \sigma^2, & \mathrm{if} \quad n=m, \\
      0, & \mathrm{otherwise},
    \end{cases}
\end{equation}
we can decompose the expectation of $\hat{s}^2$ as
\begin{align}
    & \mathbb{E}\bigl[ \hat{s}^2 \bigr] 
    = \frac{1}{N} \sum_{n=1}^N \bigl( (1 - c_n)^2 \sigma^2 + \sum_{m \neq n} c_m^2 \sigma^2 + \bigl( 1 - \sum_{m=1}^N c_m \bigr)^2 \mu^2 \bigr) \notag \\
    & = \frac{1}{N} \sum_{n=1}^N \Bigl( (1 - 2 c_n) \sigma^2 + \sum_{m=1}^N c_m^2 \sigma^2 \Bigr) + \Bigl( 1 - \sum_{m=1}^N c_m \Bigr)^2 \mu^2 \notag \\
    & = \Bigl( 1 - \frac{2}{N} \sum_{n=1}^N c_n + \sum_{n=1}^N c_n^2 \Bigr) \sigma^2 
    + \Bigl( 1 - \sum_{n=1}^N c_n \Bigr)^2 \mu^2.
\end{align}
Hence, if the coefficients $c_1, \ldots, c_N$ satisfy
\begin{equation}
    \left\{
    \begin{aligned}
    & 1 - \sum_{n=1}^N c_n = 0, \\
    & 1 - \frac{2}{N} \sum_{n=1}^N c_n + \sum_{n=1}^N c_n^2 = 1,
    \end{aligned}
    \right.
\end{equation}
then $\hat{s}^2$ is an unbiased estimator of $\sigma^2$. Rewriting these yields
\begin{equation}
    \label{eq:k2}
    \left\{
        \begin{aligned}
    & \sum_{n=1}^N c_n = 1, \\
    & \sum_{n=1}^N c_n^2 = \frac{2}{N}.
        \end{aligned}
    \right.
\end{equation}
If the coefficients satisfy Equation \eqref{eq:k2}, then $\hat{X}$ is also an unbiased estimator of $\mu$.

We can confirm that $\bar{X}_2$ in Equation \eqref{eq:half-var} corresponds to the case
\begin{equation}
    \left\{
        \begin{aligned}
    & c_1 = \cdots = c_{N/2} = \frac{2}{N}, \\
    & c_{N/2+1} = \cdots = c_N = 0,
        \end{aligned}
    \right.
\end{equation}
whose coefficients indeed satisfy the condition \eqref{eq:k2}. Other choices of coefficients can also generate an AAUV. For example, let $M$ be an integer with $1 \le M < N$, and define
\begin{equation}
   \label{eq:cM}
    \left\{
        \begin{aligned}
    & c_1 = \cdots = c_M = \frac{M + \sqrt{M(N-M)}}{NM}, \\
    & c_{M+1} = \cdots = c_N = \frac{N-M - \sqrt{M(N-M)}}{N(N-M)}.
        \end{aligned}
    \right.
\end{equation}
These satisfy \eqref{eq:k2} as well. In particular, if $M = N/2$, the coefficients reproduce $\bar{X}_2$ and $\bar{\sigma}_2^2$.

Note that a coefficient $c_n$ can be negative, but its range is constrained. From the Cauchy--Schwarz inequality
\begin{equation}
    \left( \sum_i a_i b_i \right)^2 \le \left( \sum_i a_i^2 \right) \left( \sum_i b_i^2 \right),
\end{equation}
we obtain
\begin{equation}
    \left( 1 - c_N \right)^2 = \left( \sum_{n=1}^{N-1} c_n \right)^2 \le (N-1) \sum_{n=1}^{N-1} c_n^2 = (N-1)\left( \frac{2}{N} - c_N^2 \right).
\end{equation}
Equality holds if $c_1 = \cdots = c_{N-1}$. Hence, $c_N$ must lie within
\begin{equation}
    \frac{1-\sqrt{N-1}}{N} \;\le\; c_N \;\le\; \frac{1+\sqrt{N-1}}{N}.
\end{equation}
Another characteristic follows from considering the sum of pairwise products of the coefficients:
\begin{equation}
    \sum_{n=1}^N \sum_{m \neq n} c_n c_m = \left( \sum_{n=1}^N c_n \right)^2 \;-\; \sum_{n=1}^N c_n^2 = \frac{N-2}{N}.
\end{equation}

While these findings demonstrate that there exist numerous ways to construct an average-adjusted unbiased variance estimator, their practical appeal remains limited. In general, among unbiased quadratic-form estimators of the variance, the usual unbiased variance exhibits the smallest variance \cite{hsu1938best}. Because AAUVs also boil down to quadratic forms in the sample, they cannot achieve a lower variance than the standard unbiased variance. From a computational standpoint, the $\bar{X}_2$ example illustrates that using fewer observations to compute the mean only halves the number of additions for that portion of the calculation, resulting in negligible overall speed gains. Moreover, the fact that AAUVs are inherently non-symmetric functions of the sample implies that they cannot be minimum-variance unbiased estimators, since such optimality requires symmetry \cite{rao1973linear}. Indeed, permuting the sample values can yield as many as $N!$ distinct estimates from a single AAUV formula, yet their average coincides exactly with the usual unbiased variance.

\begin{theorem}
\label{th:mvue}
Let $\hat{s}^2(X_1, ..., X_N)$ be an average-adjusted unbiased variance computed from the sample $X_1, ..., X_N$. Suppose $i_{1}, ..., i_{N}$ is a permutation of the indices $1, ..., N$. Then the average
\begin{equation}
    Q = \frac{1}{N!} \sum_{i_{1}, ..., i_{N}} \hat{s}^2\left(X_{i_1}, ..., X_{i_N}\right)
\end{equation}
taken over all such permutations is equal to the usual unbiased variance $s^2$.
\end{theorem}

\begin{proof}
We can write $\hat{s}^2$ as
\begin{align}
    & \hat{s}^2 = \frac{1}{N} \sum_{n=1}^N \left((1 - c_n) (X_n - \mu) - \sum_{m \neq n} c_m (X_m - \mu)\right)^2 \notag \\
    &= \frac{1}{N} \sum_{n=1}^N \left(
       (1 - c_n)^2 (X_n - \mu)^2
       - 2 \sum_{m \neq n} (1 - c_n)c_m (X_n - \mu)(X_m - \mu) \right. \notag \\
    &\qquad \qquad \left. + \sum_{m \neq n} c_m^2 (X_m - \mu)^2
       + \sum_{m \neq n} \sum_{m' \neq n, m} c_m c_{m'} (X_m - \mu)(X_{m'} - \mu)
    \right) \notag \\
    &= \frac{1}{N} \sum_{n=1}^N (1 - 2 c_n + N c_n^2)(X_n - \mu)^2
    - \frac{2}{N} \sum_{n=1}^N \sum_{m \neq n} (c_m - c_n c_m)(X_n - \mu)(X_m - \mu) \notag \\
    &\qquad \qquad + \frac{N-2}{N} \sum_{n=1}^N \sum_{m \neq n} c_n c_m (X_n - \mu)(X_m - \mu). \label{eq:aauv-expanded}
\end{align}
$Q$ is obtained as the average of the above expression over all permutations of the indices:
\begin{align}
    &Q = \frac{1}{N^2} \sum_{k=1}^N (1 - 2 c_k + N c_k^2) \sum_{n=1}^N (X_n - \mu)^2 \notag \\
    &\qquad \qquad - \frac{2}{N^2(N-1)} \sum_{k=1}^N \sum_{l \neq k} (c_k - c_l c_k) \sum_{n=1}^N \sum_{m \neq n} (X_n - \mu)(X_m - \mu) \notag \\
    &\qquad \qquad + \frac{N-2}{N^2(N-1)} \sum_{k=1}^N \sum_{l \neq k} c_n c_m \sum_{n=1}^N \sum_{m \neq n} (X_n - \mu)(X_m - \mu) \notag \\
    &= \frac{1}{N} \sum_{n=1}^N (X_n - \mu)^2
    - \frac{2}{N^2(N-1)}\left(N-1 - \frac{N-2}{N}\right)\sum_{n=1}^N \sum_{m \neq n} (X_n - \mu)(X_m - \mu) \notag \\
    &\qquad \qquad + \frac{N-2}{N^2(N-1)} \frac{N-2}{N} \sum_{n=1}^N \sum_{m \neq n} (X_n - \mu)(X_m - \mu) \notag \\
    &= \frac{1}{N} \sum_{n=1}^N (X_n - \mu)^2 
    + \frac{(N-2)^2 - 2N^2 + 2N + 2N - 4}{N^3(N-1)} \sum_{n=1}^N \sum_{m \neq n} (X_n - \mu)(X_m - \mu) \notag \\
    &= \frac{1}{N} \sum_{n=1}^N (X_n - \mu)^2 
    - \frac{1}{N(N-1)} \sum_{n=1}^N \sum_{m \neq n} (X_n - \mu)(X_m - \mu). \label{eq:M}
\end{align}
On the other hand, using Equation \eqref{eq:aauv-expanded} gives another form of the usual unbiased variance:
\begin{align}
    &s^2 = \frac{1}{N-1} \sum_{n=1}^N \left(1 - \frac{2}{N} + \frac{1}{N}\right)(X_n - \mu)^2 - \frac{2}{N-1}\sum_{n=1}^N \sum_{m \neq n} \frac{N-1}{N^2}(X_n - \mu)(X_m - \mu) \notag \\
    &\qquad \qquad + \frac{N-2}{N-1} \sum_{n=1}^N \sum_{m \neq n} \frac{1}{N^2} (X_n - \mu)(X_m - \mu) \notag \\
    &= \frac{1}{N}\sum_{n=1}^N (X_n - \mu)^2 
    - \frac{1}{N(N-1)} \sum_{n=1}^N \sum_{m \neq n} (X_n - \mu)(X_m - \mu),
\end{align}
which indeed equals $Q$.
\end{proof}

\section{Between the Standard Unbiased Variance and the Average-Adjusted Unbiased Variance}

Although the two estimators differ in how the mean is calculated and how the overall factor is chosen, we now consider whether there exists an intermediate unbiased estimator between them. First, define
\begin{equation}
    \tilde{X}_\lambda = \lambda \hat{X} + (1-\lambda) \bar{X}.
\end{equation}
Clearly, $\tilde{X}_0 = \bar{X}$ and $\tilde{X}_1 = \hat{X}$. Consider the sum of squared deviations from $\tilde{X}_\lambda$,
\begin{align}
    & S(\lambda) = \sum_{n=1}^N (X_n - \tilde{X}_\lambda)^2
    = \sum_{n=1}^N \left(X_n - (\lambda \hat{X} + (1-\lambda) \bar{X}) \right)^2 \notag \\
    & = \sum_{n=1}^N \left( X_n - \hat{X} - (1-\lambda) (\bar{X} - \hat{X}) \right)^2 \notag \\
    & = \sum_{n=1}^N (X_n - \hat{X})^2 - 2 (1-\lambda) (\bar{X} - \hat{X})\sum_{n=1}^N (X_n - \hat{X}) + (1-\lambda)^2 N (\hat{X} - \bar{X})^2 \notag \\
    & = \sum_{n=1}^N (X_n - \hat{X})^2 - 2 (1-\lambda) \frac{1}{N} \left( \sum_{n=1}^N (X_n - \hat{X}) \right)\left( \sum_{n=1}^N (X_n - \hat{X}) \right) \notag \\
    &\qquad \qquad + (1-\lambda)^2 N \frac{1}{N^2} \left( \sum_{n=1}^N (X_n - \hat{X}) \right)^2 \notag \\
    & = \sum_{n=1}^N (X_n - \hat{X})^2 - \frac{1-\lambda^2}{N} \left( \sum_{n=1}^N (X_n - \hat{X}) \right)^2.
\end{align}
Note that
\begin{align}
    & \mathbb{E} \left[ \left( \sum_{n=1}^N (X_n - \hat{X}) \right)^2 \right]
    = \mathbb{E} \left[ \left( \sum_{n=1}^N \left( (1-c_n) (X_n - \mu) - \sum_{m\neq n} c_m (X_m - \mu) \right) \right)^2 \right] \notag \\
    & = \sum_{n=1}^N \left( (1-c_n)^2 \mathbb{E}[(X_n - \mu)^2] + \sum_{m\neq n} c_m^2 \mathbb{E}[(X_m - \mu)^2] \right) 
    = \sigma^2 \sum_{n=1}^N \left( 1-2c_n + \sum_{m=1}^N c_m^2 \right) \notag \\
    & = \sigma^2 (N - 2 + 2) = N \sigma^2,
\end{align}
which implies
\begin{equation}
    \mathbb{E} \left[ S(\lambda) \right] = N \sigma^2  - (1-\lambda^2) \sigma^2 = (N - 1 + \lambda^2) \sigma^2.
\end{equation}
Hence,
\begin{equation}
    s_\lambda^2 = \frac{1}{N-1+\lambda^2} \sum_{n=1}^N (X_n - \tilde{X}_\lambda)^2
\end{equation}
satisfies
\begin{equation}
    \mathbb{E} [s_\lambda^2] = \sigma^2,
\end{equation}
so it is an unbiased estimator of the variance.

Thus, we derive an unbiased variance estimator that is a combination of the average-adjusted unbiased variance and the standard unbiased variance. For any $K \ge N-1$, we can construct an unbiased variance estimator with a denominator of $K$, using $\lambda = \sqrt{K-N+1}$ and any coefficients for AAUV. If we use $\hat{X}$ with coefficients from Equation \eqref{eq:cM}, then choosing $\lambda = \sqrt{(N-M)/M}$ gives
\begin{equation}
   \tilde{X}_\lambda = \frac{1}{M} \sum_{n=1}^M X_n
\end{equation}
and
\begin{equation}
   s_\lambda^2 = \frac{1}{N-1+(N-M)/M} \sum_{n=1}^N ( X_n - \tilde{X}_\lambda )^2,
\end{equation}
where the mean estimator uses only $M$ data points. As with AAUV, symmetrization recovers the standard unbiased variance.

\begin{theorem}
Let $i_1, ..., i_N$ be a permutation of the indices $1, ..., N$. Then the average
\begin{equation}
    Q(\lambda) = \frac{1}{N!} \sum_{i_1, \dots, i_N} s_\lambda^2(X_{i_1}, \dots, X_{i_N})
\end{equation}
over all permutations is equal to the standard unbiased variance $s^2$, regardless of the value of $\lambda$.
\end{theorem}

\begin{proof}
We can write
\begin{equation}
    s_\lambda^2 = \frac{1}{N-1+\lambda^2} \left( \sum_{n=1}^N (X_n - \hat{X})^2 - \frac{1-\lambda^2}{N} \left( \sum_{n=1}^N (X_n - \hat{X}) \right)^2 \right).
\end{equation}
The first term in parentheses has already been examined in the proof of Theorem \ref{th:mvue} (see Equation \eqref{eq:M}). For the square of the sum in the second term, we have
\begin{align}
    & \left( \sum_{n=1}^N (X_n - \hat{X}) \right)^2
    = \left( \sum_{n=1}^N (1 - N c_n) (X_n - \mu) \right)^2 \notag \\
    & = \sum_{n=1}^N (1 - N c_n)^2 (X_n - \mu)^2 + \sum_{n=1}^N \sum_{m \neq n} (1 - N c_n) (1 - N c_m) (X_n - \mu) (X_m - \mu).
\end{align}
When symmetrized, this becomes
\begin{align}
    & \frac{1}{N} \sum_{k=1}^N (1 - 2N c_k + N^2 c_k^2) \sum_{n=1}^N (X_n - \mu)^2 \notag \\
    & \qquad \qquad + \frac{1}{N(N-1)} \sum_{k=1}^N \sum_{l \neq k} (1 - N (c_k + c_l) + N^2 c_k c_l) \sum_{n=1}^N \sum_{m \neq n} (X_n - \mu) (X_m - \mu) \notag \\
    & = \sum_{n=1}^N (X_n - \mu)^2 \notag \\
    & \qquad + \frac{1}{N(N-1)} \left(N(N-1) - 2N (N-1) + N(N-2)\right) \sum_{n=1}^N \sum_{m \neq n} (X_n - \mu) (X_m - \mu) \notag \\
    & = \sum_{n=1}^N (X_n - \mu)^2 - \frac{1}{N-1} \sum_{n=1}^N \sum_{m \neq n} (X_n - \mu) (X_m - \mu).
\end{align}
Thus,
\begin{align}
    & Q(\lambda) = \frac{1}{N-1+\lambda^2} \left( \sum_{n=1}^N (X_n - \mu)^2 
    - \frac{1}{N-1} \sum_{n=1}^N \sum_{m \neq n} (X_n - \mu) (X_m - \mu) \right. \notag \\
    & \qquad \qquad
    \left. - \frac{1-\lambda^2}{N} \left(
    \sum_{n=1}^N (X_n - \mu)^2 - \frac{1}{N-1} \sum_{n=1}^N \sum_{m \neq n} (X_n - \mu) (X_m - \mu) 
    \right) \right) \notag \\
    & = \frac{1}{N} \sum_{n=1}^N (X_n - \mu)^2
    - \frac{1}{N(N-1)} \sum_{n=1}^N \sum_{m \neq n} (X_n - \mu) (X_m - \mu)
    = s^2,
\end{align}
which completes the proof.
\end{proof}

\section{Higher-Order Central Moments}

Having considered average-adjusted unbiased estimators for variance, a natural question is whether the same methodology can be extended to higher-order central moments. Specifically, we ask whether the $k$th central moment
\begin{equation}
    \mu_k = \mathbb{E}\left[ \left( X - \mu \right)^k \right]
\end{equation}
can be estimated unbiasedly by
\begin{equation}
    \hat{\mu}_k = \frac{1}{N} \sum_{n=1}^N \left( X_n - \hat{X} \right)^k.
\end{equation}

Let us consider the case $k=3$. Using a similar algebraic manipulation as in Equation \eqref{eq:hatmu2}, we obtain
\begin{align}
    & \hat{\mu}_3 = \frac{1}{N} \sum_{n=1}^N \left( (1 - c_n) (X_n - \mu) - \sum_{m \neq n} c_m (X_m - \mu) - \left( 1 - \sum_{m=1}^N c_m \right) \mu \right)^3 \notag \\
    & = \frac{1}{N} \sum_{n=1}^N \left( (1 - c_n) (X_n - \mu) - \sum_{m \neq n} c_m (X_m - \mu) \right)^3
    + \left( 1 - \sum_{n=1}^N c_n \right) \mu C_\mu \notag \\
    & = \frac{1}{N} \sum_{n=1}^N \left( (1 - c_n) (X_n - \mu) - \sum_{m \neq n} c_m (X_m - \mu) \right)^3 
    + \left( 1 - \sum_{n=1}^N c_n \right) \mu C_\mu,
\end{align}
where $C_\mu$ is a second-order polynomial of $\mu, X_1, ..., X_N$.
Noting that $\mathbb{E}[(X_n-\mu)(X_m-\mu)(X_l-\mu)] = 0$ unless $n=m=l$, its expectation is
\begin{align}
    & \mathbb{E}\left[ \hat{\mu}_3  \right]
    = \frac{1}{N} \sum_{n=1}^N \left( (1 - c_n)^3 \mu_3 - \sum_{m \neq n} c_m^3 \mu_3 \right) 
    - \left( 1 - \sum_{n=1}^N c_n \right) \mu \mathbb{E}\left[ C_\mu \right] \notag \\
    & = \left( 1 - \frac{3}{N} \sum_{n=1}^N c_n + \frac{3}{N} \sum_{n=1}^N c_n^2 - \sum_{n=1}^N c_n^3 \right) \mu_3
    - \left( 1 - \sum_{n=1}^N c_n \right) \mu \mathbb{E}\left[ C_\mu \right].
\end{align}
Hence, we obtain the conditions
\begin{equation}
    \left\{
        \begin{aligned}
        & 1 - \sum_{n=1}^N c_n = 0, \\
        & 1 - \frac{3}{N} \sum_{n=1}^N c_n + \frac{3}{N} \sum_{n=1}^N c_n^2 - \sum_{n=1}^N c_n^3 = 1,
        \end{aligned}
    \right.
\end{equation}
which simplify to
\begin{equation}
    \label{eq:k3}
    \left\{
        \begin{aligned}
        & \sum_{n=1}^N c_n = 1, \\
        & \frac{3}{N} \sum_{n=1}^N c_n^2 - \sum_{n=1}^N c_n^3 = \frac{3}{N}.
        \end{aligned}
    \right.
\end{equation}

We now look for one of the solutions that satisfies Equation \eqref{eq:k3}. Suppose $N$ can be written as $N = 2M + K$ for some positive integers $M$ and $K$, and assume
\begin{equation}
    \left\{
        \begin{aligned}
            & c_1 = \cdots = c_M = \alpha, \\
            & c_{M+1} = \cdots = c_{2M} = -\alpha, \\
            & c_{2M+1} = \cdots = c_N = \beta.
        \end{aligned}
    \right.
\end{equation}
It follows immediately that
\begin{equation}
    \beta = \frac{1}{K}.
\end{equation}
Moreover,
\begin{equation}
    \frac{3}{N} \sum_{n=1}^N c_n^2 - \sum_{n=1}^N c_n^3
    = \frac{6M}{N} \alpha^2 + \frac{3K}{N} \beta^2 - K \beta^3 
    = \frac{6M}{N} \alpha^2 + \frac{3K-N}{NK^2}.
\end{equation}
Thus,
\begin{align}
    & \frac{6M}{N} \alpha^2 + \frac{3K-N}{NK^2} =  \frac{3}{N}, \notag \\
    \Leftrightarrow \quad & 6MK^2 \alpha^2 = 3K^2 - 3K + N, \notag \\
    \Leftrightarrow \quad & \alpha = \sqrt{\frac{3K(K - 1) + N}{3(N-K)K^2}}.
\end{align}
This provides one solution of Equation \eqref{eq:k3}; other coefficient sets may also yield an unbiased estimator $\hat{\mu}_3$.

Here, the conditions for the third moment become relatively simple because the cross-terms in $\mathbb{E}[(X_n-\mu)(X_m-\mu)(X_l-\mu)]$ vanish unless $n=m=l$. In forth- or higher-order moments, similar expansions would introduce more complex constraints, and it remains an open question whether coefficients $c_1,\dots,c_N$ exist for every $k$ such that $\hat{\mu}_k$ is unbiased for $\mu_k$. Even if they do exist, such an estimator might not be optimal in practice, as its inherent asymmetry suggests it may have higher variance than conventional unbiased estimators.

\section{Concluding Remarks}

In this paper, we investigated the possibility of creating unbiased variance estimators by adjusting the mean estimator rather than the more familiar approach of correcting the denominator. The half-sample approach illustrated that using only part of the data to estimate the mean could still yield an unbiased variance estimator when dividing by $N$. Generalizing this idea led us to introduce average-adjusted unbiased variance (AAUV), defined by specific linear combinations of the sample that satisfy simple conditions. By employing an interpolation approach, it is also possible to construct unbiased estimators that combine both average-adjusting and denominator-adjusting. Although one can construct a variety of unbiased variance estimators in this manner, their inherent asymmetry means they cannot outperform the standard unbiased variance in terms of lower variance. While the average-adjusting approach could be extended to unbiased estimation of third- or higher-order central moments, investigating whether such solutions exist remains a topic for future research.

\section*{Disclosure statement}

The author declare no conflict of interest.

\section*{Funding}

The author declare no funds, grants, or other support were received during the preparation of this manuscript.

\section*{Acknowledgements}

ChatGPT was used for the purpose of correcting grammatical errors and improving language clarity in this manuscript.

\bigskip


\end{document}